\documentclass[12pt,reqno]{article}
\usepackage{graphicx}
\usepackage{amsmath}
\usepackage{amsthm}
\usepackage{latexsym,bm}
\usepackage{amssymb}
\usepackage{indentfirst}
\newtheorem{thm}{Theorem}[section]
\newtheorem{cor}[thm]{Corollary}
\newtheorem{lem}[thm]{Lemma}

\newtheorem{rem}[thm]{Remark}
\newtheorem{exm}{Example}

\numberwithin{equation}{section}

\title{ \bf On regular algebraic hypersurfaces with non-zero constant mean curvature in 
	Euclidean spaces}

\author{ \bf Alexandre Paiva Barreto\thanks{Work partially supported by 
		FAPESP:2018/03721-4(Brazil)}, Francisco 
	Fontenele\thanks{Work partially supported by 
		FAPESP:2019/20854-0(Brazil)} and Luiz Hartmann\thanks{Work partially 
		supported 
		by 
	FAPESP:2018/23202-1(Brazil)}}

\date{}

\begin{document}

\maketitle

\begin{quote}
\small {\bf Abstract}. 
We prove that there are no regular algebraic hypersurfaces with non-zero constant mean curvature in the 
Euclidean space $\mathbb R^{n+1}$, $n\geq 2$, defined by polynomials of odd 
degree. Also we prove that the hyperspheres and the round cylinders are the 
only regular algebraic hypersurfaces with non-zero constant mean curvature in 
$\mathbb R^{n+1}$, $n\geq 2$, defined by polynomials of degree less than or 
equal to three. These results give partial answers to a question raised by 
Barbosa and do Carmo.
\end{quote}

\begin{quote}

\end{quote}

\section{Introduction}

An algebraic hypersurface in the $(n+1)$-dimensional Euclidean space $\mathbb R^{n+1},\;n\geq 2,$ is the zero set $M=P^{-1}(0)$ of a polynomial function $P:\mathbb R^{n+1}\to\mathbb R$. We say that $M$ is {\it regular} if the gradient vector field $\nabla P$ of $P$ has no zeros on $M$. The condition of regularity implies that $M$ is a complete properly embedded hypersurface of $\mathbb R^{n+1}$. 

There are many examples of algebraic hypersurfaces in $\mathbb R^{n+1}$ which have constant mean curvature. The basic examples of such hypersurfaces are hyperplanes, hyperspheres and round cylinders. In addition to them, we have, for example, the classical Enneper and Henneberg minimal surfaces in $\mathbb R^3$ \cite{Ni}, and the families of algebraic minimal cones in $\mathbb R^{n+1}$ constructed in \cite{Tk}. For other examples, see \cite{Lo,Od,PT,S1,S2} and the references therein.
 
Barbosa and do Carmo \cite{BdC} proved that the only connected regular algebraic surfaces in $\mathbb R^3$ with non-zero constant mean curvature are the spheres and the right circular cylinders, a result that was already known for polynomials of degree less than or equal to three \cite{Pe}. For generalizations of this result for globally subanalytic CMC surfaces in $\mathbb R^3$ see \cite{BBdCF,Sa}. 

Motivated by the theorem mentioned in the previous paragraph and by the fact that the hyperspheres and the round cylinders are the only examples of regular algebraic hypersurfaces in $\mathbb R^{n+1}$ with non-zero constant mean curvature known so far, Barbosa and do Carmo \cite[p. 177]{BdC} proposed the following extension of their own result:

\vskip5pt

“The hyperspheres and the round cylinders are the only connected regular algebraic hypersurfaces in $\mathbb R^{n+1},\;n\geq 2$, with non-zero constant mean curvature.”

\vskip10pt

From Perdomo and Tkachev \cite{PT} we know that there are no regular algebraic hypersurfaces with non-zero constant mean curvature in $\mathbb R^{n+1},\;n\geq 2,$ defined by polynomials of degree $3$. Here, we prove that there can be no examples defined by polynomials of any odd degree:

\begin{thm}\label{Teo1}
Let $M^n$ be a regular algebraic hypersurface in $\mathbb R^{n+1},\;n\geq 2$, defined by a polynomial $P$ of degree $m$. If $M^n$ has non-zero constant mean curvature, then $m$ is even. 
\end{thm}

The theorem below shows that the question formulated by Barbosa and do Carmo has an affirmative answer for polynomials of degree less than or equal to three.  

\begin{thm}\label{Teo2}
Let $M^n$ be a regular algebraic hypersurface in $\mathbb R^{n+1},\;n\geq 2,$ defined by a polynomial $P$ of degree less than or equal to three. If $M^n$ has non-zero constant mean curvature, then $M^n$ is a hypersphere or a round cylinder.
\end{thm}

Using Theorem \ref{Teo1} one concludes that if $M^n\subset\mathbb R^{n+1},\;n\geq 2,$ is a regular algebraic hypersurface that has non-zero constant mean curvature and is defined by a polynomial $P$ of degree $m\leq 3$, then $m=2$, i.e., $M^n$ is a quadric hypersurface in $\mathbb R^{n+1}$. Thus, in the case $n=2$, Theorem \ref{Teo2} follows immediately from Theorem \ref{Teo1} and the well known classification of quadric surfaces in $\mathbb R^3$. However, the situation is quite different for $n\geq 3$, since, to the best of our knowledge, there is not a classification of quadric hypersurfaces in $\mathbb R^{n+1}$ for $n\geq 3$.

\section{Our arguments}

A polynomial $P:\mathbb R^{n+1}\to\mathbb R$ of degree $m$ can be expressed in 
a unique way as the sum 
\begin{equation}\label{Decomposition}
P=\sum_{i=0}^mP_i,
\end{equation}
where $P_m\neq 0$ and each $P_i$ is a homogeneous polynomial of degree $i$. We 
call $P_i,\;i=0,...,m,$ the {\it homogeneous factors} of $P$, and $P_m$ the 
{\it highest order homogeneous factor} of $P$. Since $P_m$ clearly changes 
sign when $m$ is odd, Theorem \ref{Teo1} is a consequence of the stronger theorem below. From now on we consider that a regular algebraic hypersurface $M=P^{-1}(0)$ is oriented by the global unit normal vector field $N=\nabla P/|\nabla P|$.

\begin{thm}\label{Stronger}
Let $M^n$ be a regular algebraic hypersurface in $\mathbb R^{n+1},\;n\geq 2$, 
given by a polynomial $P$ of degree $m$. If $M^n$ has 
non-zero constant mean curvature, then the highest order homogeneous factor 
$P_m$ of $P$ is semi-definite, i.e., either $P_m(x)\geq 0$ for all $x\in\mathbb 
R^{n+1}$ or $P_m(x)\leq 0$ for all $x\in\mathbb R^{n+1}$.
\end{thm}

\begin{proof}
Let (\ref{Decomposition}) be the expression of $P$ as the sum of homogeneous polynomials. Writing a point in $\mathbb R^{n+1}$ as $x=(x_1,...,x_{n+1})$, let $U^{+}=\{x\in\mathbb R^{n+1}:P(x)>0\}$ and $U^{-}=\{x\in\mathbb R^{n+1}:P(x)<0\}$. Since the mean curvature $H$ of $M$ is a non-zero constant by hypothesis, changing $P$ by $-P$ if necessary, we can assume that $H=c>0$. This means that the mean curvature vector $\overrightarrow{H}$ of $M^n$ points in the direction of $U^{+}$.

By (\ref{Decomposition}), for any $t\in\mathbb R$ and any vector $v$ in the unit sphere $S^n_1$ we have 
\begin{eqnarray}\label{Eq8}
P(tv)=P_m(v)t^m+P_{m-1}(v)t^{m-1}+\cdots+P_1(v)t+P_0,
\end{eqnarray}
and so
\begin{eqnarray}\label{Eq9}
\frac{P(tv)}{t^m}-P_m(v)=\sum_{i=0}^{m-1}\frac{P_i(v)}{t^{m-i}},\;\;\;\;v\in S^n_1,\;t\neq 0.
\end{eqnarray}
Using (\ref{Eq9}) and the compactness of $S^n_1$, one easily sees that $t^{-m}P(tv)\to P_m(v)$ uniformly on $S^n_1$ when $t\to\infty$.

Suppose, by contradiction, that $P_m$ changes sign. Then, since $P_m$ is homogeneous, there exists a vector $w$ in the unit sphere $S^n_1\subset\mathbb R^{n+1}$ such that $P_m(w)>0$. Hence, by continuity, there is a closed disk $W$ around $w$ in $S^n_1$ such that $P_m(v)>P_m(w)/2$ for all $v\in W$. Since $t^{-m}P(tv)\to P_m(v)$ uniformly on $S^n_1$, there exists $t_0>0$ such that $|t^{-m}P(tv)-P_m(v)|<P_m(w)/4$ for all $t>t_0$ and $v\in S^n_1$. Combining these informations, one obtains   
\begin{eqnarray}\label{Eq9b}
t^{-m}P(tv)&\geq& -|t^{-m}P(tv)-P_m(v)|+P_m(v)\nonumber\\&>&-\frac{1}{4}P_m(w)+\frac{1}{2}P_m(w)=\frac{1}{4}P_m(w)>0,\;\;\;t>t_0,\;v\in W,
\end{eqnarray}
and so  
\begin{eqnarray}\label{Eq9d}
\{tv:t>t_0,\;v\in W\}\subset U^{+}.
\end{eqnarray}

Let $R$ be an arbitrary positive number. By (\ref{Eq9d}), there exists a ball $B$ of radius $R$ in $\mathbb R^{n+1}$ such that $B\subset U^{+}$. Let $x_0$ be the center of $B$ and $f:M\to\mathbb R$ the function defined by $f(x)=||x-x_0||^2$. Since $M$ is a closed subset of $\mathbb R^{n+1}$, there is a point $p\in M$ such that $f(p)=\inf_M f$. Let $v$ be an arbitrary unit vector of $T_pM$ and $\gamma:I\to M$ a smooth curve such that $\gamma(0)=p$ and $\gamma'(0)=v$. Since the function $f(\gamma(t))$ attains a minimum at $t=0$, one has
\begin{eqnarray}\label{Eq3a}
0=\frac{d}{dt}\Big|_{t=0}f(\gamma(t))=\frac{d}{dt}\Big|_{t=0}\langle\gamma(t)-x_0,\gamma(t)-x_0\rangle=2\langle v,p-x_0\rangle
\end{eqnarray}
and
\begin{eqnarray}\label{Eq3b}
0\leq\frac{d^2}{dt^2}\Big|_{t=0}f(\gamma(t))=2\langle\gamma''(0),p-x_0\rangle+2.
\end{eqnarray}
From (\ref{Eq3a}) one obtains that $p-x_0$ is orthogonal to $T_pM$. Since $N$ points inward $U^+$ and $x_0\in U^{+}$, it follows that $N(p)=(x_0-p)/||x_0-p||$. Using this information in (\ref{Eq3b}), one obtains
\begin{eqnarray}\label{Eq3c}
\langle Av,v\rangle=-\langle(N\circ\gamma)'(0),\gamma'(0)\rangle=\langle N(p),\gamma''(0)\rangle\leq\frac{1}{||x_0-p||},\nonumber
\end{eqnarray}
for any unit vector $v\in T_pM$, where $A$ is the shape operator of $M$ with respect to $N$. Taking the trace in the above inequality and using the fact that $p\not\in B$, one concludes that
\begin{eqnarray}\label{Eq7}
H=H(p)\leq\frac{1}{||x_0-p||}\leq\frac{1}{R},\nonumber 
\end{eqnarray} 
for every $R>0$, contradicting $H=c>0$. This contradiction shows that $P_m$ does 
not change sign, and the theorem is proved.
\end{proof}

\begin{rem}
{\em The arguments used in the proof of Theorem \ref{Stronger} show that the complement of the zero set of any non-zero polynomial, in any number of variables, contains balls of arbitrarily large radius.}
\end{rem}

In the proof of Theorem \ref{Teo2}, as well as in the proof of Corollary \ref{CorDefinite}, we will use the following result, which is of interest in its own right. We believe this result is known, but since we were unable to find a reference in the literature, we will provide a proof for it here. 
\begin{lem}\label{Pol}
Let $P:\mathbb R^{n+1}\to\mathbb R$ be a polynomial of degree $m,\;m\geq 2$. If for some $k,\;1\leq k\leq n,$ $P(x)$ vanishes on $S_r^k(a)\times\mathbb R^{n-k}$, where $S_r^k(a)$ is the hypersphere of $\mathbb R^{k+1}$ of radius $r$ and center $a$, then $P(x)$ is divisible by the polynomial 
$$
Q(x)=\sum_{i=1}^{k+1}(x_i-a_i)^2-r^2.
$$
\end{lem}

\begin{proof}
We will prove the lemma in the case where $m=2d,\;d\geq 1$. The proof in the case that the degree of $P(x)$ is odd is entirely analogous and will be omitted. 

Assume first $a=0$ and $r=1$. Write $P$ as
\begin{eqnarray}\label{Polyn1}
P=\sum_{i=0}^dP_{2i}+\sum_{i=0}^{d-1}P_{2i+1},
\end{eqnarray}
where $P_j, j=0,...,2d,$ is a homogeneous polynomial of degree $j$. By hypothesis, for every $v\in S^k_1(0)\times\mathbb R^{n-k}$ one has 
\begin{equation*}\label{Polyn2}
\begin{aligned}
0=&P(v) =\sum_{i=0}^dP_{2i}(v)+\sum_{i=0}^{d-1}P_{2i+1}(v),\\
0=&P(-v)=\sum_{i=0}^dP_{2i}(v)-\sum_{i=0}^{d-1}P_{2i+1}(v).
\end{aligned}
\end{equation*}
From the two equalities above one obtains 
$\sum\limits_{i=0}^dP_{2i}(v)=0=\sum\limits_{i=0}^{d-1}P_{2i+1}(v)$, which 
implies 
\begin{eqnarray}\label{Polyn8}
P_{2d}(v)=-\sum_{i=0}^{d-1}P_{2i}(v),\;\;\;P_{2d-1}(v)=-\sum_{i=0}^{d-2}P_{2i+1}(v),\;\;\;\forall v\in S^k_1(0)\times\mathbb R^{n-k}.
\end{eqnarray}
Let $\{e_1,...,e_{n+1}\}$ be the canonical basis of $\mathbb R^{n+1}$. For any 
$x=(x_1,...,x_{n+1})\in\mathbb R^{n+1}$ such that 
$y:=\sum\limits_{i=1}^{k+1}x_ie_i\neq 0$ it holds that $x/|y|\in 
S^k_1(0)\times\mathbb R^{n-k}$. By (\ref{Polyn1}) and (\ref{Polyn8}), for all 
such $x$ one has  
\begin{eqnarray}\label{Polyn6}
P(x)&=&\sum_{i=0}^{d-1}P_{2i}(x)+P_{2d}(x)+\sum_{i=0}^{d-2}P_{2i+1}(x)+P_{2d-1}(x)\nonumber\\&=&\sum_{i=0}^{d-1}P_{2i}(x)+|y|^{2d}P_{2d}(x/|y|)+\sum_{i=0}^{d-2}P_{2i+1}(x)+|y|^{2d-1}P_{2d-1}(x/|y|)\nonumber\\&=&\sum_{i=0}^{d-1}P_{2i}(x)-|y|^{2d}\sum_{i=0}^{d-1}P_{2i}(x/|y|)+\sum_{i=0}^{d-2}P_{2i+1}(x)-|y|^{2d-1}\sum_{i=0}^{d-2}P_{2i+1}(x/|y|)\nonumber\\&=&\sum_{i=0}^{d-1}(1-|y|^{2(d-i)})P_{2i}(x)+\sum_{i=0}^{d-2}(1-|y|^{2(d-i-1)})P_{2i+1}(x).\nonumber
\end{eqnarray}
Notice that the second term on the right hand side of the above equality vanishes when $d=1$. Since $(1-|y|^{2k})$ is divisible by $1-|y|^2$ for any integer $k\geq 1$, it follows from the above equality that 
\begin{eqnarray}\label{Polyn6a}
P(x)=(|y|^2-1)R(x), 
\end{eqnarray}
for some polynomial $R(x)$ of degree $2(d-1)$. This proves the lemma in the 
case $a=0$ and $r=1$, since $|y|^2=\sum\limits_{i=1}^{k+1}x_i^2$. To obtain the 
lemma in the general case, it suffices to apply (\ref{Polyn6a}) for the 
polynomial $\overline{P}(x):=P(rx+a)$.
\end{proof}

Theorem \ref{Stronger} states that if a regular algebraic hypersurface in 
$\mathbb R^{n+1}$ has non-zero constant mean curvature, then the highest order 
homogeneous factor $P_m$ of its defining polynomial $P$ is semi-definite (in 
particular, the degree $m$ of $P$ is even). Before proving Theorem \ref{Teo2}, 
for completeness let us say what happens in the case where $P_m$ is definite, 
i.e. $P_m(x)\neq 0$ for all $x\in\mathbb R^{n+1}-\{0\}$:

\begin{thm}\label{Definite}
Let $M^n$ be a regular algebraic hypersurface in $\mathbb R^{n+1},\;n\geq 2$, 
defined by a polynomial $P$ of degree $m$. If the highest order homogeneous 
factor $P_m$ of $P$ is definite, then $M^n$ is compact. In particular, if $M^n$ 
has constant mean curvature, then $M^n$ is a finite union of hyperspheres of 
$\mathbb R^{n+1}$. 
\end{thm}

\begin{proof}
Let (\ref{Decomposition}) be the expression of $P$ as the sum of its 
homogeneous factors. Since $P_m$ is definite by hypothesis, changing $P$ by 
$-P$ if necessary we can assume that $P_m(v)>0$ for all $v\in S^n_1$. Let 
$\alpha=\inf\{P_m(v):v\in S^n_1\}>0$. By the proof of Theorem \ref{Stronger}, 
$$
t^{-m}P(tv)\overset{unif}{\longrightarrow} P_m(v) \;\;\text{on}\;\; S^n_1\;\;\text{as}\;\;t\to\infty.
$$
Then there exists $t_0>0$ such 
that $|t^{-m}P(tv)-P_m(v)|<\alpha/2$, for all $t>t_0$ and $v\in S^n_1$, and so
\begin{eqnarray}\label{Eq9c}
t^{-m}P(tv)&\geq& -|t^{-m}P(tv)-P_m(v)|+P_m(v)>-\frac{\alpha}{2}+\alpha=\frac{\alpha}{2}>0,\nonumber
\end{eqnarray} 
for all $t>t_0$ and $v\in S^n_1$. Hence, $P(tv)>0$ for all $t>t_0$ and $v\in 
S^n_1$, which implies that the set $U^{-}=\{x\in\mathbb R^{n+1}:P(x)<0\}$ is 
bounded. Since $M=\partial (U^{-})$, one concludes that $M$ is compact. 

Assume now that $M^n$ has constant mean curvature. Since $M^n$ is embedded, by a well known theorem of Alexandrov \cite{Al} each connected component of $M^n$ is a hypersphere of $\mathbb R^{n+1}$.
\end{proof}

The following result is an immediate consequence of Lemma \ref{Pol} and Theorem \ref{Definite}.

\begin{cor}\label{CorDefinite}
Let $M^n$ be a regular algebraic hypersurface in $\mathbb R^{n+1},\;n\geq 2$, 
defined by an irreducible polynomial $P$ of degree $m$. If the highest order homogeneous factor 
$P_m$ of $P$ is definite and $M^n$ has constant mean curvature, then $M^n$ is a 
hypersphere of $\mathbb R^{n+1}$. 
\end{cor}

\begin{exm}
{\em Let $P:\mathbb R^3\to\mathbb R$ be the polynomial defined by   
$$
P(x,y,z)=\big(x^2+y^2+z^2-1\big)\big((x-3)^2+y^2+z^2-1\big).
$$
It is easy to see that the gradient $\nabla P$ of $P$ vanishes nowhere in 
$M^2:=P^{-1}(0)$. Then $M^2$ is a regular algebraic surface in $\mathbb R^3$. 
The highest order homogeneous factor of $P$ is 
$P_4=x^4+y^4+z^4+2x^2y^2+2x^2z^2+2y^2z^2$, which is clearly definite. Being the 
union of two disjoint unit spheres of $\mathbb R^3$, $M^2$ has constant mean 
curvature. This shows that the hypothesis in Corollary \ref{CorDefinite} that 
$P$ is irreducible cannot be dropped.} 
\end{exm}

\noindent{\it Proof of Theorem \ref{Teo2}}. Since, by hypothesis, the degree $m$ of $P$ is less than or equal to $3$ and $M^n$ has non-zero constant mean curvature, it follows from Theorem \ref{Teo1} that $m=2$. Hence, we can express $P$ as  
\begin{eqnarray}\label{Decomposition2}
P=P_0+P_1+P_2,
\end{eqnarray}
where $P_i,\;i=0,1,2$, is a homogeneous polynomial of degree $i$ in the 
variables $x_1,...,x_{n+1}$. Changing $P$ by $-P$ if necessary, we can assume that
$H=c>0$. We will establish the theorem by proving that 

\vskip10pt

\noindent ($\dagger$)\;\;Up to a rigid motion of $\mathbb R^{n+1}$, $M^n=S^k\times\mathbb R^{n-k}$\; for some $k,\;1\leq k\leq n$, where $S^k$ is a hypersphere of $\mathbb R^{k+1}$.

\vskip10pt

We will prove ($\dagger$) by induction on $n$. If $n=2$, it follows from the well known classification of quadric surfaces in $\mathbb R^3$ and from the hypothesis that $M$ has non-zero constant mean curvature that $M$ is a sphere or a right circular cylinder. This shows that ($\dagger$) holds for $n=2$.

Assume now $n\geq 3$ and that ($\dagger$) holds for $n-1$. If $P_2$ is definite, then $M^n$ is a finite union of hyperspheres of $\mathbb R^{n+1}$ by Theorem \ref{Definite}. Since the degree of $P$ is two, it follows from Lemma \ref{Pol} that $M^n$ is a single hypersphere of $\mathbb R^{n+1}$. Thus ($\dagger$) holds for $n$ in the case that $P_2$ is definite.

If $P_2$ is indefinite, then there exists $w\in S^n_1$ such that $P_2(w)=0$. Then, after a change of coordinates given by an orthogonal transformation that sends $e_{n+1}$ to $w$, the polynomial P can be written as 
\begin{eqnarray}\label{Decomposition3}
P(x_1,...,x_n,x_{n+1})=A_1(x_1,...,x_n)x_{n+1}+A_0(x_1,...,x_n),
\end{eqnarray}
where $A_1(x_1,...,x_n)$ is a polynomial of degree $\leq 1$ and $A_0(x_1,...,x_n)$ a polynomial of degree $\leq 2$. 

\vskip10pt

\noindent{\bf Claim.} $A_1=0$.

\vskip10pt

Assume that the claim is not true. Since the degree of $A_1$ is at most $1$, given $R>0$ there exists a closed ball $B\subset\mathbb R^n$ of radius $R$ on which $A_1(x_1,...,x_n)\neq 0$. If $A_1>0$ on $B$, using (\ref{Decomposition3}) one easily verifies that $P(x_1,...,x_n,x_{n+1})>0$ for all $(x_1,...,x_n)\in B$ and all $x_{n+1}>\max_B |A_0|/\min_B A_1$. Similarly, if $A_1<0$ on $B$ one has $P(x_1,...,x_n,x_{n+1})>0$ for all $(x_1,...,x_n)\in B$ and all $x_{n+1}<\max_B |A_0|/\max_B A_1$. In either case, one sees that there is a ball of radius R in $\mathbb R^{n+1}$ entirely contained in the set $U^{+}=\{x\in\mathbb R^{n+1}:P(x)>0\}$. Then, by an argument used in the proof of Theorem \ref{Stronger}, $H\leq 1/R$ for every $R>0$, contradicting $H=c>0$. This proves the claim.

\vskip5pt

By (\ref{Decomposition3}) and the claim above, $P(x_1,...,x_n,x_{n+1})=A_0(x_1,...,x_n)$, where $A_0:\mathbb R^n\to\mathbb R$ is a polynomial of degree two in the variables $x_1,...,x_n$. Since 0 is a regular value for $P$, so is for $A_0$. Hence $M^n=M_0^{n-1}\times\mathbb R$, where $M_0^{n-1}=A_0^{-1}(0)$ is a regular algebraic hypersurface of $\mathbb R^n$. Since $M^n$ has non-zero constant mean curvature, the same is true of $M_0^{n-1}$. Then, by the induction hypothesis, up to a rigid motion of $\mathbb R^n$ one has $M_0^{n-1}=S^k\times\mathbb R^{n-1-k}$, for some $k,\;1\leq k\leq n-1$, where $S^k$ is a hypersphere of $\mathbb R^{k+1}$. Hence $M^n=S^k\times\mathbb R^{n-k}$, for some $k,\;1\leq k\leq n-1$. This shows that ($\dagger$) holds for $n$ also in the case that $P_2$ is indefinite. Then, by the induction principle, ($\dagger$) holds for every $n\geq 2$, and the theorem is proved.\qed

Alexandre Paiva Barreto
\vskip1pt
Departamento de Matem\'atica
\vskip1pt
Universidade Federal de S\~ao Carlos
\vskip1pt
S\~ao Carlos, SP, Brazil
\vskip1pt
alexandre@dm.ufscar.br;
alexandre.paiva@ufscar.br

\vskip15pt

Francisco Fontenele
\vskip1pt
Departamento de Geometria
\vskip1pt
Universidade Federal Fluminense
\vskip1pt
Niter\'oi, RJ, Brazil
\vskip1pt
fontenele@mat.uff.br

\vskip15pt

Luiz Hartmann
\vskip1pt
Departamento de Matem\'atica
\vskip1pt
Universidade Federal de S\~ao Carlos 
\vskip1pt
S\~ao Carlos, SP, Brazil
\vskip1pt
luizhartmann@ufscar.br;
hartmann@dm.ufscar.br


\begin{thebibliography}{s2}

\bibitem{Al} A.D. Aleksandrov, {\em Uniqueness theorems for surfaces in the large}, Vestnik
Leningrad Univ. Math. {\bf{13}} (1958), 5-8.

\bibitem{BBdCF} J.L. Barbosa, L. Birbrair, M. do Carmo and A. Fernandes, {\em Globally subanalytic CMC surfaces in $\mathbb R^3$}, Electron. Res. Announc. Math. Sci. {\bf{21}} (2014), 186-192.

\bibitem{BdC} J.L. Barbosa and M. do Carmo, {\em On regular algebraic surfaces of $\mathbb R^3$ with constant mean curvature}, J. Differential Geom. {\bf{102}} (2016), 173-178.

\bibitem{Lo} R. L\'opez, {\em Surfaces with constant mean curvature in Euclidean space}, Int. Electron. J. Geom. {\bf 3} (2010), 67-101.

\bibitem{Ni} J.C.C. Nitsche, {\em Lectures on Minimal Surfaces}, Vol. 1, Cambridge University
Press, Cambridge (1989). Introduction, fundamentals, geometry and basic
boundary value problems. Translated from the German by Jerry M. Feinberg,
With a German foreword.

\bibitem{Od} B. Odehnal, {\em On algebraic minimal surfaces}, KoG {\bf 20} (2016), 61-78.

\bibitem{Pe} O. Perdomo, {\em Algebraic constant mean curvature surfaces in Euclidean space}, Houston J. Math. {\bf{39}} (2013), 127-136.

\bibitem{PT} O. Perdomo and V.G. Tkachev, {\em Algebraic CMC hypersurfaces of order 3 in Euclidean spaces}, J. Geom. {\bf{110}} (2019), 5-10.

\bibitem{Sa} J.E. Sampaio, {\em Globally subanalytic CMC surfaces in $\mathbb R^3$ with singularities}, Proc. Roy. Soc. Edinburgh Sect. A {\bf{151}} (2021), 407-424.

\bibitem{S1} A. Small, {\em Algebraic minimal surfaces in $\mathbb R^4$}, Math. Scand. {\bf{94}} (2004), 109-124.

\bibitem{S2} A. Small, {\em On algebraic minimal surfaces in $\mathbb R^3$ deriving from charge 2 monopole
spectral curves}, Internat. J. Math. {\bf{16}} (2005), 173-180.

\bibitem{Tk} V. Tkachev, {\em Minimal cubic cones via Clifford algebras}, Complex Anal. Oper.
Theory {\bf{4}} (2010), 685-700.

\end{thebibliography}
\end{document}